\documentclass[11pt]{amsart}

\usepackage{url}
\usepackage{hyperref}%make all the references and links clickable

\usepackage[all,cmtip]{xy}
\usepackage{graphicx}
\usepackage{amssymb,amsmath,amscd,graphicx,
latexsym,amsthm}
\usepackage{amssymb,latexsym,amsmath,amscd,graphicx}
\usepackage[mathscr]{eucal}
 \input xy

\usepackage[letterpaper, centering]{geometry}

 \xyoption{all}

\def\N{{\mathbb N}} 
\def\Q{{\mathbb Q}}
\def\C{{\mathbb C}}

\def\OO{{\mathcal O}}

\def\FF{\mathcal{F}}

\def\II{\mathcal{I}}

\def\Pic{{Pic}}
\def\Pic0{\mathrm{Pic}^0}

\def\Inf{\mathrm{Inf}}

\def\deg{\mathrm{deg}}

\def\NS{\mathrm{NS}}

\def\rank{\mathrm{rk\,}}

\def\la{{\langle}}
\def\ra{{\rangle}}

\theoremstyle{plain}

\newtheorem{theorem}{Theorem}[section]
\newtheorem{theoremalpha}{Theorem}

\newtheorem{propositionalpha}[theoremalpha]{Proposition}
\newtheorem{specialcasealpha}[theoremalpha]{Special Case}
\newtheorem{conjecturealpha}[theoremalpha]{Conjecture}

\newtheorem{proposition/example}[theorem]{Proposition/Example}
\newtheorem{definition/theorem}[theorem]{Definition/Theorem}
\newtheorem{proposition}[theorem]{Proposition}

\newtheorem{corollary}[theorem]{Corollary}
\newtheorem{lemma}[theorem]{Lemma}

\newtheorem{criterion}[theorem]{Criterion}

\theoremstyle{definition}
\newtheorem{definition}[theorem]{Definition}

\newtheorem{remark}[theorem]{Remark}

\newtheorem{conjecture}[theorem]{Conjecture}
\newtheorem{conjecture/question}[theorem]{Conjecture/Question}

\newtheorem{remark/definition}[theorem]{Remark/Definition}
\newtheorem{notation/assumptions}[theorem]{Assumptions/Notation}

\numberwithin{equation}{section}

 \theoremstyle{remark}

\begin{document}  

\title{Kobayashi hyperbolicity of complete linear systems on abelian varieties}

 \author{Federico Caucci}
\address{Sapienza Universit\`a di Roma, Dipartimento di Scienze di Base e Applicate per l'Ingegneria, Via Antonio Scarpa 16, 00161 Roma, Italy}
 \email{federico.caucci@uniroma1.it}
 \thanks{The author was partially supported by the FinQuaAV project No.\! RP125199A4E9CD68, CUP B83C25004290005,  as part of the ``Ricerca Scientifica di Ateneo 2025'' program  of Sapienza University of Rome.}

\maketitle

%\tableofcontents
\setlength{\parskip}{.1 in}

\begin{abstract} 
We provide  a geometric condition ensuring that a very general element of a complete linear system on an abelian variety is Kobayashi hyperbolic.    
This can be seen as  
 the counterpart of some classical results on very general hypersurfaces in projective spaces. 
Some related conjectures are also given, and involve a natural constant, the so-called basepoint-freeness threshold, associated with any polarized abelian variety $(A, L)$  detecting information about the hyperbolicity of $|L|$.
 \end{abstract}

\section{Introduction}

It is well-known that a complex abelian variety $A$ is not  Kobayashi hyperbolic, as it admits non-constant entire curves, that is, non-constant holomorphic maps $\C \to A$. 
 Besides, by the so-called Bloch's Theorem, an entire curve $f \colon \C \to A$ is subject to a very strong restriction: the Zariski closure $\overline{f(\C)}$ is a translation of an abelian subvariety of $A$. This result has a long history. It was stated in 1926 by Bloch  \cite{bloch} with an incomplete proof.    A complete argument was given only around the end of the 70's, independently and at the same time, by several people:  
 Ochiai \cite{ochai} and Kawamata \cite{kawamata}, Smyth \cite{smyth}, Green-Griffiths \cite{greengriffiths},  and Wong \cite{wong}. All these proofs basically build on the original approach of Bloch.  
McQuillan \cite{mcquillan} later found a different,  more arithmetic argument.  
 
It is also known that the complement of an ample divisor on an abelian variety is Brody hyperbolic, as conjectured by Lang and proved by Siu-Yeung \cite{siu} (see also \cite{yamanoi1} for a stronger and more recent result), that is, the image  $f(\C)$ of a non-constant entire curve $f \colon \C \to A$ must intersect the support of any ample divisor on $A$. As usual, one refers to a variety which admits no non-constant entire curves by saying that it is Brody hyperbolic. This condition is equivalent to the Kobayashi hyperbolicity for \emph{compact} complex varieties by Brody's theorem \cite{brody}. In general, Kobayashi hyperbolicity implies Brody's one. 
We point the interested reader to \cite{kobayashi, diverio}, and to \cite{demailly, demailly2, javan} for surveys on Kobayashi hyperbolicity and related topics. 

Very recently, the following result appeared:
\begin{theorem}[\cite{itoetal}, Corollary 1.10]\label{thmI}
Let $A$ be a complex abelian variety of dimension $g \geq 3$, and let $L$ be an ample line bundle on $A$. Then,
\begin{enumerate}
\item[\emph{1)}] The linear system $|mL|$ is algebraically hyperbolic, if $m \geq g$.

\item[\emph{2)}] If $|L|$ has no base divisors, then $m \geq g-1$ suffices. 
\end{enumerate}
\end{theorem}
\noindent A basepoint-free linear system is said to be algebraically (resp.\! Kobayashi) hyperbolic, if a very general element is algebraically (resp.\! Kobayashi) hyperbolic as a variety (let us recall that, as soon as $m \geq 2$,
 $|mL|$ is always basepoint-free on an abelian variety. See, e.g., \cite[Proposition 4.1.5]{birk}). 

Algebraic hyperbolicity is a notion  introduced by Demailly in \cite{demailly} as an algebraic analogue of Kobayashi hyperbolicity. The reader may find in \S \ref{S2}  its definition and the (conjectural) relation with Kobayashi hyperbolicity. 
Although not noted in \cite{itoetal},  it is known that these two notions are equivalent for closed subvarieties of abelian varieties (see Remark \ref{rmkjava} below).   
Theorem \ref{thmI} confirms, in the case of abelian varieties and with actually a better bound, a general conjecture of Moraga and Yeong \cite{moraga} about the algebraic hyperbolicity of adjoint bundles (see Conjecture \ref{conjectureMY} below), which has been motivated by Fujita's famous conjecture \cite{fuconj}.

Let us recall that, thanks to the Decomposition Theorem (see \cite[\S 4.3]{birk}, or \cite[\S 10.3]{kempf}),  the base locus of an ample line bundle $L$ on an abelian variety $A$ contains a divisor if and only if the polarized abelian variety $(A, L)$ contains a principally polarized factor $(B, \Theta)$, i.e., 
\[
(A, L) \simeq (B, \Theta) \times (A', L')
\] 
as polarized abelian varieties. 
As observed  in \cite{itoetal}, Item \emph{(1)} of Theorem \ref{thmI} is optimal:  if 
\begin{equation}\label{special1}
(A, L) \simeq (E, \Theta) \times  (A', L')\, ,
\end{equation}
 with $(E, \Theta)$ a principally polarized \emph{elliptic curve}, then any element of $|mL|$ is not hyperbolic when $m\leq g-1$, and $g = \dim A \geq 2$ (see \emph{op.cit.}, Proof of Proposition 1.12). 
We point out that polarized abelian varieties as in \eqref{special1} have been characterized by Nakamaye \cite{nakamaye}, and, more
recently, by Alvarado-Pareschi \cite{alpa2}.

In this note, we aim to show how taking into account the \emph{geometry} of a polarized abelian variety $(A, L)$ leads to an improvement of Theorem \ref{thmI}, in such a way  that \eqref{special1} turns out to be the only ``exceptional'' case. Namely, we have:
\begin{theoremalpha}\label{main}
Let $A$ be a complex abelian variety of dimension $g \geq 3$, and let $L$ be an ample line bundle on $A$. Then, $|mL|$ is Kobayashi hyperbolic if 
\[
m \geq g-1\, ,
\]
unless $(A, L)$ is as in \eqref{special1}. In this case, one needs to take $m \geq g$.
\end{theoremalpha}
\noindent The proof of Theorem \ref{main} uses 
recent results of 
 Pareschi \cite{parANT, parPA} and Alvarado-Pareschi \cite{alpa1, alpa2} about  certain generation properties of coherent sheaves on abelian varieties, 
along with 
a criterion in \cite{itoetal}, which guarantees the algebraic hyperbolicity of a basepoint-free linear system on $A$ (we state this criterion directly for abelian varieties as Criterion \ref{criterion} below, but it holds true more generally for varieties with nef tangent bundle. This is used in \cite{itoetal} to get the algebraic hyperbolicity of linear systems in several other cases). 

In particular, we get 
\begin{specialcasealpha}\label{specialcase}
Let $(A, \Theta)$ be an indecomposable principally polarized abelian variety of dimension $g \geq 3$. Then, the linear system $|m\Theta|$ is Kobayashi hyperbolic if $m \geq g-1$.
\end{specialcasealpha}
\noindent Here, \emph{indecomposable} means that $(A, \Theta)$ cannot be written as the product of two lower-dimensional p.p.a.v.'s $(A_i, \Theta_i)$, which, according to the above-mentioned Decomposition Theorem, is equivalent to saying that the theta divisor $\Theta$ is irreducible.

Concerning the sensibly simpler
  case of abelian surfaces,   we have the next result, which is included for completeness.
\begin{propositionalpha}\label{propsurfaces}
Let $A$ be a complex abelian variety of dimension $g =2$, and let $L$ be an ample line bundle on $A$. Then, $|mL|$ is Kobayashi hyperbolic as soon as 
$m \geq 2$, and $|L|$ itself is Kobayashi hyperbolic,
 if $L$ is basepoint-free.
\end{propositionalpha}
\noindent In the second arXiv version of \cite{itoetal}, the authors extended this result to dimensions $g \geq 3$ by proving that the linear system $|L|$ is hyperbolic, if $L$ separates $(g-2)$-jets at any point of $A$  (see \cite[Remark 1.9]{itoetal2}. Recall that $0$-jets separation means basepoint-freeness). They also give an alternative proof of our Theorem \ref{main} along these lines (see \cite[Theorem 5.6]{itoetal2}).

Given  now a (possibly primitive) polarized abelian variety  $(A, L)$  of dimension $g \geq 1$, a constant  $\beta(A, L)$ attached to it, which is called the basepoint-freeness threshold and has been introduced in \cite{jipa}, recently proved to detect a lot of information about the positivity of $(A, L)$ (see \cite{jipa, ca, ji1}, to name just a few). In \cite{ca}, the author conjectured that $\beta(A, L)$ could be bounded from above in terms of  the degree, with respect to $L$, of all non-zero abelian subvarieties of $A$ (see \eqref{conj000}). This would provide a higher syzygies analogue, for abelian varieties, of the generalized Fujita's conjecture. Along with the hyperbolicity criterion of \cite{itoetal} (Criterion \ref{criterion} below) and \cite[Theorem 1.9]{itoetal}, this suggests a conjecture  
about  the hyperbolicity of a (possibly primitive) linear system $|L|$ on an abelian variety $A$:
\begin{conjecturealpha}\label{conjIntro}
\emph{Let $(A, L)$ be a polarized abelian variety of dimension $g \geq 2$, such that
\begin{equation}\label{degree0}
(L^{\dim B} \cdot B) > ((g-1)\dim B)^{\dim B}
\end{equation}
for all abelian subvarieties $\{0\} \neq B \subseteq A$. Then, $|L|$ is Kobayashi hyperbolic.} 
\end{conjecturealpha} 
\noindent We note that part of the above-mentioned conjecture of \cite{ca} predicts that, under the assumption \eqref{degree0}, $L$ is indeed basepoint-free (and actually very ample, when $g \geq 3$). We will show in \S \ref{lastS} that Conjecture \ref{conjIntro} is true in dimension $g \leq 3$, and ``almost true'' if $g =4$ (see Theorem \ref{jiangThm4}(1)). Moreover, it holds in arbitrary dimension ``up to a factor $2$'' (see Proposition \ref{propIntro00}).   
We refer the reader to \S \ref{lastS} for more on this circle of ideas. 
Here, we just collect the next result for a \emph{general} member of the moduli space of polarized abelian varieties of a given type, which   follows from the semi-continuity
of the basepoint-freeness threshold and from certain known upper bounds:  
\begin{propositionalpha}\label{proplast0}
Let $(A, L)$  be a general polarized abelian variety of dimension $g \geq 2$. If
\[
h^0(A, L) \geq \frac{(2g(g-1))^g}{g!}\, , 
\]
then $|L|$ is Kobayashi hyperbolic.
\end{propositionalpha} 
\noindent This already improves \cite[Corollary 1.13]{itoetal}, and Proposition \ref{proplast1} below gives  an even better (but more involved) result. 
Moreover, 
 the following weaker version of Conjecture \ref{conjIntro}  holds true when $g \leq 6$, and it has not yet been verified
  only for a  finite number of  primitive polarization types in each dimension $g \geq 7$:  
\begin{conjecturealpha}
Let $(A, L)$ be a general polarized abelian variety of dimension $g \geq 2$.  If
\begin{equation*}
h^0(A, L) > \frac{(g\, (g-1))^g}{g!}\, ,
\end{equation*}
 then $|L|$ is Kobayashi hyperbolic.
\end{conjecturealpha}

Beauville \cite{beauville} recently observed that a complete linear system $|L|$ on an abelian variety $A$ has maximal variation,  provided that $L$ is  an ample and globally generated line bundle on $A$. Maximal variation here means that, if $D$ is a general member in $|L|$, there are only finitely many divisors $D' \in |L|$ which are isomorphic to $D$. Thereby, 
the results  of the present paper  yield a whole bunch of non-trivial examples of non-isomorphic hyperbolic subvarieties of $A$.

\vskip0.3truecm\noindent\textbf{Acknowledgment.} 
I would like to thank Beppe Pareschi for his interest, for answering my questions about the proof of Lemma \ref{lemmaAP}, and for fruitful exchanges, 
 Ariyan Javanpeykar for    Remark \ref{rmkjava}  
 and for his kind guidance through the hyperbolicity literature, and Atsushi Ito for useful communications. 
 The author is a member of GNSAGA-INdAM.

\section{Algebraic hyperbolicity}\label{S2}

 As mentioned in the Introduction, the next definition has been introduced by Demailly \cite{demailly} (see also \cite{jaka}).
\begin{definition}
A  complex projective variety $X$ is said to be \emph{algebraically hyperbolic} if there exist a  real number $\varepsilon > 0$ and an ample line bundle $L$ on $X$ such that, for every  smooth projective curve $C$ and every non-constant morphism 
\[
f \colon C \to X\, ,
\]
 one has that
\[
2 g(C) - 2 \geq \varepsilon \,  \deg_C (f^* L)\, ,
\]
where $g(C)$ is the genus of $C$.
\end{definition}
\noindent In particular, an algebraically hyperbolic variety does not contain any rational and elliptic curves. Demailly \cite{demailly} proved   that a Kobayashi hyperbolic projective variety is algebraically hyperbolic, and conjectured that also the converse implication holds true. This is, in general, quite open. 

Since we are particularly interested in subvarieties of abelian varieties, let us promptly observe the following. 
\begin{remark}\label{rmkjava}
Let  $Y$ be a closed subvariety of an abelian variety $A$.  The following are equivalent:
\begin{enumerate}
\item[1)] $Y$ is Kobayashi hyperbolic;
\item[2)] $Y$ is Brody hyperbolic;
\item[3)] $Y$ is algebraically hyperbolic;
\item[4)] $Y$ does not contain translates of positive-dimensional abelian subvarieties of $A$.  
\end{enumerate}
Indeed, (1) and (2) are equivalent thanks to Brody's theorem \cite{brody}. Item (1) implies (3) by Demailly's theorem \cite[Theorem 2.1]{demailly} (see also the comment below \cite[Theorem 1.2]{jaka}). The implication (3) $\Rightarrow$ (4) is, for instance, implicit in \cite{jaka}. Indeed, \cite[Corollary 4.5]{jaka} says that an algebraically hyperbolic variety  is groupless, that is, it does not admit any non-trivial map from an algebraic group. So, in particular, one gets (4).  
Finally,  (4) implies (2) thanks to the Bloch's theorem mentioned in the Introduction, saying that the Zariski closure of the image of an entire curve in $A$ is a translation of an abelian subvariety of $A$. The equivalence between (1) and (4) was also proved by Green in \cite{green}.
\end{remark}

Quite naturally, an ample and basepoint-free complete linear system $|L|$ on $X$ is said to be \emph{algebraically hyperbolic}, if a very general element of $|L|$ is algebraically hyperbolic. A similar definition may be given with the \emph{Kobayashi hyperbolicity}. Moraga and Yeong conjectured  what follows.
\begin{conjecture}[\cite{moraga}, Conjecture 1.1]\label{conjectureMY}
Let $L$ be an ample line bundle on a smooth projective variety $X$ of dimension $n \geq 2$. Then, the linear system 
\[
|K_X + m L|
\] is algebraically hyperbolic as soon as
$m \geq 3n+1$.
\end{conjecture}
\noindent They  have been motivated by Fujita's conjecture \cite{fuconj}, and by both classical and more recent results on hyperbolicity of (very) general hypersurfaces in projective spaces (see \emph{op.cit.}, the references therein, and \cite{diverio}). They proved their conjecture for smooth projective toric varieties
(see also \cite{haesong} for results about spherical varieties). More recently,  \cite{itoetal} produced a useful criterion to detect the algebraic hyperbolicity of basepoint-free  complete linear systems on a variety $X$ with  \emph{nef} tangent bundle (and possibly mild singularities). We state it here directly when $X = A$ is an  abelian variety, and refer the reader to \cite[Theorem 4.7]{itoetal} for the more general statement.  Its proof  builds on a nowadays standard deformation type argument,  recently simplified and further developed  by Coskun and Riedl in \cite{coskun1, coskun2}. 
\begin{criterion}[Particular case of \cite{itoetal}, Theorem 4.7]\label{criterion}
Let $A$ be a complex abelian variety of dimension $g \geq 3$. Let $N$ and $L$ be line bundles on $A$, with $N$ globally generated and $L$ ample. If there exists a rational number $\delta > 0$ such that
\begin{enumerate}
\item[\emph{i)}] $M_{N} \la \delta L \ra$ is nef, and \\

\item[\emph{ii)}] $N - \delta(g-2) L$ is ample, 
\end{enumerate} 
then $|N|$ is an algebraically hyperbolic linear system on $A$.
\end{criterion}  
\noindent  Let us recall that $M_{N}$ is the \emph{kernel bundle} of $N$, that is, the kernel of the evaluation morphism of global sections of $N$
\[
0 \to M_N \to H^0(A, N) \otimes \OO_A \to N \to 0\, ,
\] 
and $M_{N} \la \delta L \ra$ denotes the $\Q$-twisted bundle in the sense of Lazarsfeld \cite[\S 6.2A]{laII}.

We plan to apply the above Criterion \ref{criterion} and Remark \ref{rmkjava} to get Theorem \ref{main}. With this in mind, the notions and results of the next section are useful. Before moving on to it, let us consider the easy case of hyperbolicity of linear systems on abelian surfaces.
\begin{proof}[Proof of Proposition \ref{propsurfaces}]
Let $A$ be an abelian surface, and $L$ be an ample and globally generated line bundle on $A$. Take $D \in |L|$ general, so that $D$ is a smooth projective irreducible curve. 
By taking the holomorphic Euler characteristic 
of its defining short exact sequence 
\[
0 \to \OO_A(-D) \to \OO_A \to \OO_D \to 0\, ,
\]
 we get
\begin{align*}
0 = \chi(A, \OO_A) &= \chi(D, \OO_D) + \chi(A, \OO_A(-D)) \\
&= 1 - g(D) + h^2(A, \OO_A(-D))\, ,
\end{align*}
where $g(D)$ is the genus of $D$. 
Therefore, we have
\[
g(D) = 1 + h^0(A, \OO_A(D)) > 1\, ,
\]
as $D$ is an ample (hence effective) divisor on $A$. 
This means that $D$ is a hyperbolic curve.

If $L$ is not globally generated, we may apply the above reasoning to $|mL|$, which is basepoint-free when $m\geq 2$.
 This concludes the proof of Proposition \ref{propsurfaces}.
\end{proof}

\section{Coherent sheaves on abelian varieties: regularity and generation properties}\label{S3}

Let $(A, L)$ be a polarized abelian variety of dimension $g \geq 1$. 
Given a coherent sheaf $\FF$ on $A$, we denote by 
\[
V^i(A, \FF) := \{ \alpha \in \Pic0 A \ |\ h^i(A, \FF \otimes \alpha) \neq 0 \}
\] 
the $i$\emph{-th} cohomological support loci. These are  Zariski closed subsets of $\Pic0 A$ by upper-semicontinuity. 
\begin{definition}
A coherent sheaf 
$\FF$ on $A$
is said to be \emph{GV} (resp.\! \emph{M-regular}) if 
\[
\mathrm{codim}_{\Pic0 A} V^i(A, \FF) \geq i \quad (\mathrm{resp.} > i)
\]
for all $i > 0$.
\end{definition}
Such notions turn out to be quite useful in the study of linear systems on abelian varieties (see \cite{ppI}), and 
 have been extended to the $\Q$-twisted setting by Jiang-Pareschi \cite{jipa} as follows: given a $\Q$-twisted sheaf $\FF\la r L \ra$ on $A$, with $r = \frac{a}{b} \in \Q$ and $b > 0$, we say that  $\FF\la r L \ra$ is \emph{GV} (resp.\! \emph{M-regular}) if so is $\mu_b^*(\FF) \otimes L^{ab}$, where 
\[
\mu_b \colon A \to A
\]
is the multiplication-by-$b$ isogeny of $A$. The definition makes sense because $\mu_b^*L$ is algebraically equivalent to $L^{b^2}$, and hence $L^{ab}$ plays the r\^ole of ``$\mu_b^*(\frac{a}{b}L)$''.

Theorems of Debarre \cite{debarre} and Pareschi-Popa \cite{ppIII} say that an $M$-regular sheaf is ample, and a $GV$ sheaf is nef, and it is clear that these implications remain valid in the $\Q$-twisted setting (see, e.g., \cite[Lemma 3.7]{itoetal}). Note also that the converse implications are not true, unless we are dealing with line bundles.
 As observed in \cite[Remark 3.10]{itoetal},  if $m \geq 2$, one has that the $\Q$-twisted sheaf
\begin{equation}\label{twistedM0}
M_{mL} \la \frac{m}{m-1} L \ra
\end{equation}
is  $GV$, and hence it is nef. This is a consequence of the computations performed in \cite[\S 8]{jipa}, involving the Fourier-Mukai-Poincar\'e transform. 
However,  \eqref{twistedM0} may not be $M$-regular, and actually    it is so precisely when $L$ has no base divisors. Indeed, by \cite[Proposition 8.1]{jipa} (see also \cite[Proposition 4.1]{ito0}), one has that \eqref{twistedM0} is $M$-regular if and only if
\[
\II_p \la \frac{1}{m} mL \ra = \II_p \otimes L
\] 
is $M$-regular, where $\II_p$ is the ideal sheaf of a closed point $p \in A$, and this is equivalent to asking that $L$ has no base divisors (see \cite[Example 3.10(3)]{ppsurvey}). 

Let us now explain, for the sake of completeness and self-containedness, how \cite{itoetal} obtained Theorem \ref{thmI} from the above results, 
 by applying 
 their Criterion \ref{criterion} with $N = mL$.  
\begin{proof}[Proof of Theorem \ref{thmI}]
 Fix an integer $m \geq 2$. As we explained above, 
\begin{equation}\label{proofthmI1}
M_{mL}\la \frac{m}{m-1} L \ra 
\end{equation}
is nef, and it is ample if  $L$ has no base divisors. 
Let us assume we are in the latter case, and take a rational number 
 $\delta < \frac{m}{m-1}$ such that $M_{mL}\la \delta L \ra$ is still ample. 
So, we only need to check \emph{(ii)} of the Criterion \ref{criterion}. 
Since
\begin{equation}\label{proofthmI0}
m - \delta (g-2) > m - \frac{m}{m-1} (g-2),    
\end{equation}
one has that $mL - \delta(g-2)L = (m - \delta(g-2))L$ is ample as soon as the right-hand side of \eqref{proofthmI0} is $\geq 0$, and this happens
if and only if $m \geq g-1$. 

If our line bundle $L$ has some base divisor, then we still have that \eqref{proofthmI1} is nef, and we may take $\delta = \frac{m}{m-1}$. In this case, the Criterion \ref{criterion} applies as soon as $(m - \delta(g-2))L$ is ample, that is, as soon as 
\[ 
m - \frac{m}{m-1} (g-2) > 0\, ,
\]
i.e.,  $m \geq g$. 
\end{proof}

\begin{remark}\label{rmkthmI}
More generally, but with essentially the same argument,  \cite[Theorem 1.9]{itoetal} proves   that for a polarized abelian variety $(A, L)$ of dimension $g \geq 3$ and such that 
\begin{equation}\label{rmk00}
\II_p \la \frac{1}{g-1} L\ra \quad \mathrm{is\, \emph{M-}regular,}
\end{equation}
 one has that $|L|$ is   hyperbolic. Indeed, \eqref{rmk00} gives, in particular, that $L$ is globally generated, and one still has in this case, again by \cite{jipa}, that \eqref{rmk00} is equivalent to the $M$-regularity of  $M_{L}\la \frac{1}{g-2} L \ra$. Then, the proof proceeds   in the same way as before. 
\end{remark}

As already observed (see \eqref{special1}), Theorem \ref{thmI}\emph{(1)} is optimal. 
However, 
we note that the above proof leaves  room for improvement, as it is used that the $\Q$-twisted sheaf \eqref{proofthmI1} is $M$-regular (and hence ample), while what is really needed is just its   ampleness. In this regard, we recall the following 
\begin{definition}[\cite{parPA}]\label{defPar}
Let $A$ be an abelian variety  of dimension $g \geq 1$, and $Z$ be an irreducible subvariety of $\Pic0 A$. 
A coherent sheaf $\FF$ on $A$ is said to be \emph{generated by} $Z$ if the sum of twisted evaluation maps
\[
\mathrm{ev}_U \colon \bigoplus_{\alpha \in U} H^0(A, \FF \otimes \alpha) \otimes \alpha^{\vee} \to \FF
\] 
is surjective for every nonempty open subset $U \subseteq \Pic0 A$  meeting $Z$. 
\end{definition}
\noindent  A generated sheaf is nef, as it can be written as the quotient of a finite direct sum of numerically trivial line bundles (see \cite[Remark 2.1.4]{parPA}). A theorem of Debarre \cite{debarre} can be stated, using this language, by saying that an $M$-regular sheaf on $A$ is generated by the full $\Pic0 A$ (this is what is commonly called being continuously globally generated). Finally, let us note that a globally generated sheaf is generated by $\{\hat{0}\} \subseteq \Pic0 A$, but the converse does not hold in general.

\noindent  In \cite{parPA}, Pareschi gave an ampleness criterion for generated sheaves (we state below a rather special case, which suffices for our purposes).   Let us say that an irreducible subvariety $Z \subseteq \Pic0 A$ \emph{spans} $\Pic0 A$ if there exists an integer  $N \geq 1$ such that
\[
\underbrace{Z + \ldots + Z}_{N} = \Pic0 A\, .
\] 
\begin{theorem}[\cite{parPA}, Theorem A]\label{parcriterion}
Let $\FF$ be a coherent sheaf on $A$. If $\FF$ is generated by an irreducible subvariety spanning $\Pic0 A$, then $\FF$ is ample.
\end{theorem}

So, going back to our previous notations, we observed that, when $m \geq 2$, 
\[
M_{mL}\la \frac{m}{m-1} L \ra
\]
 is always $GV$, and now the problem is to understand \emph{when} it is generated by a spanning subvariety of $\Pic0 A$.

\section{Proof of Theorem \ref{main}}

Let us start by considering a polarized abelian variety $(A, L)$ of dimension $g \geq 1$, for the moment. 
We need to introduce some more notations. In \cite{parANT, alpa1}, the authors noted (and used) that the $\Q$-twisted sheaf $\OO_A \la \lambda L \ra$, with $\lambda \in \Q$, cohomologically behaves  as a certain ``usual''  vector bundle, denoted by $E_{A, \lambda L}$. 
This is a simple semi-homogeneous vector bundle on $A$ such that 
\[
\frac{\det E_{A, \lambda L}}{\rank E_{A, \lambda L}} = \lambda L 
\]
holds in the N\'eron-Severi group $\NS(A)_{\Q}$ of $A$. Its existence was proved  by Mukai \cite[\S 7]{mukai1} (see also \cite[\S 1.5]{alpa1}). It has the following crucial property \cite[Proposition 2.1.1]{alpa1}: 
\begin{equation}\label{crupb}
\mu_b^* E_{A, \frac{a}{b}L} \simeq ({L'}^{ab})^{\oplus \, r_{A, L}(\frac{a}{b})}\, ,     
\end{equation}
where $L'$ is a line bundle on $A$, algebraically equivalent to $L$, and $r_{A, L}(\frac{a}{b}) := \rank E_{A, \frac{a}{b} L}$. 
 This fact tells us that, for a $\Q$-twisted sheaf $\FF \la \frac{a}{b} L \ra$, considering the sheaf $\mu_b^* (\FF) \otimes L^{ab}$ as in \S \ref{S3} is not the only option to deal with an ``integral'' sheaf, instead of a $\Q$-twisted one. Another, cohomologically equivalent, possibility is to look at $\FF \otimes E_{A, \frac{a}{b}L}$, because 
\[
\mu_b^* (\FF \otimes E_{A, \frac{a}{b} L}) = (\mu_b^* \FF \otimes {L'}^{ab})^{\oplus \, r_{A, L}(\frac{a}{b})}\, .
\] 
We refer the reader to
 \cite{alpa1}, where  the theory coming from this observation is fully developed.

With this terminology in our hands, we may now recall a recent result of Alvarado-Pareschi \cite{alpa2}, which has its roots in \cite{parANT}.
\begin{lemma}[\cite{alpa2}]\label{lemmaAP}
Let $(A, \Theta)$ be an \emph{indecomposable} p.p.a.v.\! of dimension $g \geq 2$. Then, for any integer $m \geq 2$,
the sheaf 
\begin{equation*}\label{eqlemmaAP}
M_{m\Theta} \otimes E_{A, \frac{m}{m-1} \Theta}
\end{equation*}
 is generated by an ample prime divisor $D$ of $\Pic0 A$. 
\end{lemma}
\noindent 
This Lemma  is proved in 
 \cite[Lemma 3.2.1]{alpa2} with $m = 2$ (where $E_{A, 2 \Theta} = 2 \Theta$), and 
in \cite[Lemma 5.2.2]{alpa2} for all $m \geq 2$, suitably generalizing the same argument.

\begin{corollary}\label{corAP}
Under the same assumptions,  the $\Q$-twisted sheaf $M_{m\Theta} \la \frac{m}{m-1}\Theta \ra$ is ample.
\end{corollary}
\begin{proof}
Let us first prove that $M_{m\Theta} \otimes E_{A, \frac{m}{m-1} \Theta}$ is an ample bundle. 
Thanks to the above Lemma \ref{lemmaAP} and Theorem  \ref{parcriterion}, it suffices to show that the  generating ample prime divisor $D$  spans $\Pic0 A$. Taking cohomologies in
\[
0 \to \OO_{\Pic0 A}(-D) \to \OO_{\Pic0 A} \to \OO_{D} \to 0
\] 
gives
\begin{equation}\label{+1}
H^1(\OO_{\Pic0 A}(-D)) \to H^1(\OO_{\Pic0 A}) \to H^1(\OO_{D})\, .
\end{equation}
Since $h^1(\OO_{\Pic0 A}(-D)) = h^{g-1}(\OO_{\Pic0 A}(D)) =  0$,
one gets from \eqref{+1} that 
\[
h^1(\OO_D) \geq h^1(\OO_{\Pic0 A}) = g > g-1 = \dim D\, .
\]
 Therefore, $D$ cannot be a translation of an abelian subvariety of $\Pic0 A$.  Let us denote by 
\[
a_n \colon \underbrace{D \times \ldots \times D}_{n} \to \Pic0 A
\]
the addition morphism 
\[
a_n\left((\alpha_1, \ldots, \alpha_n)\right) = \alpha_1 + \ldots + \alpha_n\, ,
\]
and by $D_n := \mathrm{Im}(a_n)$.  Up to translation, we may assume that $\hat{0} \in D$. So we have the inclusions
\[
D = D_1 \subseteq D_2 \subseteq \ldots \subseteq D_n \subseteq D_{n+1} \subseteq \ldots\, .
\]
Since each $D_n$ is irreducible, there exists an $n_0$ such that 
\[
D_{n_0} = D_{n_0 + 1} = \ldots\, . 
\]
It turns out that this $D_{n_0}$ is an abelian subvariety of $\Pic0 A$ (see \cite[pp.\! 109-110]{ueno}), such that
\[
D \subsetneq D_{n_0} \subseteq \Pic0 A\, . 
\] 
 It follows thereby   that, for dimensional reasons, $D_{n_0} = \Pic0 A$, i.e., 
 $D$ spans $\Pic0 A$.

Now, thanks to \eqref{crupb}, we have
\begin{equation}\label{eqlast0}
\mu_{m-1}^*(M_{m\Theta} \otimes E_{A, \frac{m}{m-1} \Theta}) \simeq \mu_{m-1}^*(M_{m\Theta}) \otimes ({{\Theta}'}^{\, m(m-1)})^{\oplus \, r_{A,  \Theta}(\frac{m}{m-1})}\, . 
\end{equation}
Therefore, by standard properties on ampleness of vector bundles, we get that 
\[
\mu_{m-1}^*(M_{m\Theta}) \otimes {\Theta}^{\, m(m-1)} = \mu_{m-1}^* (M_{m\Theta} \la \frac{m}{m-1}\Theta \ra)
\]
is ample. Since $\mu_{m-1} \colon A \to A$ is finite and surjective, one has that $M_{m\Theta} \la \frac{m}{m-1}\Theta \ra$ is ample as well (see \cite[Lemma 6.2.8]{laII}).
\end{proof}

\begin{proof}[Proof of the special case \ref{specialcase}]
It follows at once from the above Corollary \ref{corAP} and the Criterion \ref{criterion}, as in the proof of Theorem \ref{thmI}. Namely, we take an indecomposable p.p.a.v.\! $(A, \Theta)$ of dimension $g \geq 3$. Thanks to Corollary \ref{corAP}, there exists a rational number 
$0 < \delta < \frac{m}{m-1}$, where $m \geq 2$, such that 
$M_{m\Theta} \la \delta \Theta \ra$ is ample. Since, as soon as $m \geq g-1$, one has $m - \delta (g-2) > m - \frac{m}{m-1}(g-2) \geq 0$,  we get that
$(m - \delta(g-2))\Theta$ is ample. Therefore,  we may apply the Criterion \ref{criterion} -- and Remark \ref{rmkjava} -- to $|m\Theta|$, if $m \geq g-1$.  
\end{proof}

We are finally ready to give the 
\begin{proof}[Proof of Theorem \ref{main}]
Let $(A, L)$ be a polarized abelian variety of dimension $g \geq 3$, that is not isomorphic, as a polarized abelian variety, to a product 
of a principally polarized elliptic curve with another polarized abelian variety 
as in \eqref{special1}. 
 By the Decomposition Theorem (\cite[\S 4.3]{birk}, or \cite[\S 10.3]{kempf}), we have that
\[
(A, L) \simeq (B, N) \times (A_1, \Theta_1) \times \ldots \times (A_k, \Theta_k)\, ,
\]
where $N$ is an ample line bundle without base divisors on an abelian variety $B$, and the $(A_i, \Theta_i)$'s are indecomposable p.p.a.v.'s of dimension $\geq 2$. 

Thanks to the discussion below \eqref{twistedM0}, one has that, if $m \geq 2$, $M_{mN} \la \frac{m}{m-1} N \ra$ is $M$-regular, and hence ample,  on $B$.
On the other hand, thanks to Corollary \ref{corAP}, 
$M_{m\Theta_i} \la \frac{m}{m-1} \Theta_i \ra$ is ample on $A_i$ for all $i = 1, \ldots, k$. 
Let $0 < \delta < \frac{m}{m-1}$ be a rational number such that all these $\Q$-twisted bundles are still ample.  
Then, it follows easily (see, e.g., \cite[Proposition 3.4]{itoetal}) that
\[
M_{m(N \, \boxtimes\, \Theta_1 \, \boxtimes\,  \ldots \, \boxtimes \, \Theta_k)} \la \delta (N\boxtimes \Theta_1 \boxtimes \ldots \boxtimes \Theta_k)\ra
\simeq M_{mL} \la \delta L \ra
\]
is a nef $\Q$-twisted bundle. Since 
 \[
m  - \delta (g-2) > m - \frac{m}{m-1}(g-2) \geq 0\, ,
\] 
where the second inequality holds true as we are assuming $m \geq g-1$, we get that 
\[
mL - \delta(g-2)L
\] is ample. Hence, it remains only  to apply  Criterion \ref{criterion}  and Remark \ref{rmkjava}. 
\end{proof}

\section{The basepoint-freeness threshold and hyperbolicity of linear systems}\label{lastS}
Let $(A, L)$ be a polarized abelian variety of dimension $g \geq 1$, and $\II_p$ be the ideal sheaf of a closed point $p \in A$. Then, the \emph{basepoint-freeness threshold}  of $(A, L)$ is
\[
\beta(A, L) := \Inf \{r \in \Q_{> 0} \ |\ \II_{p}\la r L \ra\,  \textrm{is \emph{GV}} \}\, .
\]
It is known that $L$ is basepoint-free if and only if $\beta(A, L) < 1$, and 
 moreover, if $\beta(A, L) < r_0$, then $\II_{p}\la r_0 L \ra$ is $M$-regular (see \cite[p.\! 842 and Theorem 5.2]{jipa}). Hence, we have
\begin{remark}\label{rmkS5}
If \[
\beta(A, L) < \frac{1}{g-1}\, ,
\]
 with $g \geq 2$, then
$|L|$ is     hyperbolic. Indeed, 
this follows from Proposition \ref{propsurfaces} when $g = 2$, and from \cite[Theorem 1.9]{itoetal} (see also Remark \ref{rmkthmI} above) when $g \geq 3$. 
\end{remark}

On the other hand, it is expected by \cite{ca} that the following geometric upper bound holds 
\begin{equation}\label{conj000}
\beta(A, L) \leq \mathrm{Max}_{B} \, \frac{\dim B}{\sqrt[\dim B]{((L^{\dim B} \cdot B)}}\, ,
\end{equation}
where $B$ varies among all non-zero abelian subvarieties of $A$.\footnote{
Note that  we really  have a Max, and not only a Sup, in \eqref{conj000}. This is because the set  
\[
\left[ \frac{g}{\sqrt[g]{(L^g)}}, + \infty \right)\, \bigcap \, \left\{\frac{d}{\sqrt[d]{n}} \, |\, 1 \leq d \leq g-1, n \in \N \right\}
\]
is finite (I thank A.\! Ito for this nice observation which simplifies our previous argument). }  
The inequality in \eqref{conj000}
has been completely verified for abelian surfaces \cite{ca} and   for abelian $3$-folds \cite{itoAG},   
  ``almost'' for abelian $4$-folds \cite[Theorem 1.6]{ji2}, and in arbitrary dimension 
   ``up to a factor $2$'' \cite[Theorem 1.4]{ji1}. Namely,  
\begin{theorem}\label{jiangThm4}
\begin{enumerate}
\item[\emph{1)}\cite{ji2}] If $g = 4$, one has
\[
\beta(A, L) \leq  \mathrm{Max} \{\frac{4.31}{\sqrt[4]{((L^4)}}, \frac{\dim B}{\sqrt[\dim B]{((L^{\dim B} \cdot B)}} \ | \ \{0\} \neq B \subsetneq A \}\, ;
\]

\item[\emph{2)}\cite{ji1}] When $g$ is arbitrary, 
\begin{equation}\label{conj001}
\beta(A, L) \leq \mathrm{Max}_{B} \, \frac{2\, \dim B}{\sqrt[\dim B]{((L^{\dim B} \cdot B)}}\, .
\end{equation}
\end{enumerate}
\end{theorem}
\noindent  By requiring that the right-hand side of \eqref{conj001} is $< \frac{1}{g-1}$, from Remark \ref{rmkS5} one  obtains: 
\begin{proposition}\label{propIntro00}
Let $(A, L)$ be a polarized abelian variety of dimension $g \geq 2$.  If
\begin{equation*}\label{degree1}
(L^{\dim B} \cdot B) > (2\, (g-1)\dim B)^{\dim B}
\end{equation*}
for all abelian subvarieties $\{0\} \neq B \subseteq A$, then $|L|$ is Kobayashi hyperbolic.
\end{proposition}
\noindent  In the same vein, Conjecture \ref{conjIntro} of the Introduction would follow at once from \eqref{conj000}, and hence 
it is essentially verified  when $g \leq 4$.

In \cite{ji2}, Jiang also proved the following estimate for $\beta(A, L)$,  when $A$ is \emph{simple}: 
\begin{theorem}[\cite{ji2}, Theorems 1.6, 1.8 and 1.9]\label{ji2alpha} 
If $(A, L)$ is a simple polarized abelian variety of dimension $g \geq 4$, there exists an \emph{explicit} real constant $\alpha_g > 0$ such that
\begin{equation}\label{conj002}
\beta(A, L) \leq \frac{2g - \alpha_g}{\sqrt[g]{((L^{g})}}\, .
\end{equation}
\end{theorem}
\noindent Even if, on a simple abelian variety, any (basepoint-free)  complete  linear system is hyperbolic by Remark \ref{rmkjava}, such bound is useful to get results about the  hyperbolicity of linear systems on \emph{general} polarized abelian varieties. 
\begin{proposition}\label{proplast1}
Let $(A, L)$ be a general polarized abelian variety of dimension $g \geq 4$ and type $(d_1, \ldots, d_g)$.  If
\begin{equation*}\label{degree1}
h^0(A, L) > \frac{((2g- \alpha_g)(g-1))^g}{g!}\, ,
\end{equation*}
 then $|L|$ is Kobayashi hyperbolic.
\end{proposition}
\begin{proof}
By Theorem \ref{ji2alpha},  if $(A_0, L_0)$ is a simple polarized abelian variety of dimension $g$ and type $(d_1, \ldots, d_g)$, then
\[
\beta(A_0, L_0) \leq \frac{2g - \alpha_g}{\sqrt[g]{(L_0^{g})}} = \frac{2g - \alpha_g}{\sqrt[g]{g!\, h^0(A, L)}} \leq r\, ,
\]
where $r < \frac{1}{g-1}$ is a rational number whose existence is guaranteed by our assumption.
Then, 
it is not difficult to see that  
\begin{equation*}\label{conj0011}
\beta(A, L) \leq r
\end{equation*}
for a general polarized abelian variety $(A, L)$ of dimension $g$ and of type $(d_1, \ldots, d_g)$  
(indeed, the basepoint-freeness threshold is upper semi-continuous. See \cite[Theorem 3.1]{itoNachr}). 
Since $r < \frac{1}{g-1}$, we get the hyperbolicity of $|L|$ by Remark \ref{rmkS5}. 
\end{proof}
In particular, from the entire above discussion, one obtains at once Proposition \ref{proplast0} of the Introduction.  
Similar (and sometimes better) results can  also be obtained from \cite[Theorem 3.1]{ji2}. Moreover, 
the same reasoning suggests the following weaker version of Conjecture \ref{conjIntro}:
\begin{conjecture}\label{weakconj}
Let $(A, L)$ be a general polarized abelian variety of dimension $g \geq 2$.  If
\begin{equation*}
h^0(A, L) > \frac{(g\, (g-1))^g}{g!}\, ,
\end{equation*}
 then $|L|$ is Kobayashi hyperbolic.
\end{conjecture}
\noindent Thanks to \cite[Theorem 1.5]{ji2}, Conjecture \ref{weakconj} holds true when $g \leq 6$, and there are  only finitely many primitive polarization types in each dimension $g \geq 7$ for which the conjecture has not yet been verified
 (see \cite[\S 3]{ji2} for details).

Finally, let us recall that, by \cite[Lemma 3.4(i)]{itoAG}, for any polarized abelian variety $(A, L)$ of dimension $g$, one has that 
\[
\frac{1}{\sqrt[g]{h^0(A, L)}} \leq \beta(A, L)\, ,
\]
and, 
 when $(A, L)$ is general in its moduli space, it is expected (see \cite[p.\! 3]{ji2}) that $\beta(A, L)$ is quite close to $\frac{1}{\sqrt[g]{h^0(A, L)}}$. 
Hence, sometimes even better estimates than the one in Conjecture \ref{weakconj} may be  possible.

\providecommand{\bysame}{\leavevmode\hbox
to3em{\hrulefill}\thinspace}

\end{document}